\documentclass{amsproc}%\pagestyle{plain}

\usepackage{graphicx, xcolor}
\usepackage{comment}

\usepackage{latexsym,amsmath,amsthm,amsfonts,mathrsfs,MnSymbol,epsfig}
\usepackage[initials]{amsrefs}
\usepackage[all]{xy}
\usepackage{bm}
\usepackage{xcolor}

\usepackage{stackengine}
\stackMath
\newcommand{\con}{\boldsymbol{\nabla}}
\newcommand{\curv}{\boldsymbol{\theta}}

\providecommand{\CC}{{\mathbb{C}}}
\providecommand{\RR}{{\mathbb{R}}}

\providecommand{\ZZ}{{\mathbb{Z}}}

\providecommand{\HH}{{\mathcal H}}

\providecommand{\mcB}{{\mathcal{B}}}
\providecommand{\mcC}{{\mathcal{C}}}

\providecommand{\mcE}{{\mathcal{E}}}

\providecommand{\mcS}{{\mathcal{S}}}
\providecommand{\mcW}{{\mathcal{W}}}

\providecommand{\msL}{{\mathscr{L}}}

\providecommand{\Sch}{{\mathcal{S}}}

\providecommand{\Symb}{{\mathscr{S}}}

\providecommand{\Sp}{{\mathrm{Sp}}}

\newcommand{\ang}[1]{\langle #1 \rangle} % <...>

\newcommand{\Ext}{\mathscr{E}}
\DeclareMathOperator{\Hom}{Hom}
\DeclareMathOperator{\End}{End}
\DeclareMathOperator{\Ch}{Ch}
\DeclareMathOperator{\Tch}{Tch}

\DeclareMathOperator{\Tr}{Tr}
\DeclareMathOperator{\tr}{tr}
\DeclareMathOperator{\re}{Re}

\DeclareMathOperator{\Res}{Res}
\DeclareMathOperator{\Ker}{Ker}

\DeclareMathOperator{\ind }{Index}

\DeclareMathOperator{\Op}{Op}

\DeclareMathOperator{\Trh}{{\overline{\mathrm{Tr}}}}
\newcommand{\ho}{\mathcal{H}}
\newcommand{\iso}{\mu}

\newtheorem{theorem}{Theorem}[section]
\newtheorem{lemma}[theorem]{Lemma}

\newtheorem{proposition}[theorem]{Proposition}

\theoremstyle{definition}
\newtheorem{definition}[theorem]{Definition}

\theoremstyle{remark}

\numberwithin{equation}{section}

\title[Cyclic cohomology and the extended Heisenberg calculus]{Cyclic cohomology and the extended Heisenberg calculus of Epstein and Melrose }
\author{Alexander Gorokhovsky}
\address{University of Colorado, Boulder, Campus Box 395, Boulder, Colorado, 80309, USA}
\email{gorokhov@colorado.edu}
\author{Erik van Erp}
\address{Dartmouth College, 6188, Kemeny Hall, Hanover, New Hampshire, 03755, USA}
\email{jhamvanerp@gmail.com}
\begin{document}

\begin{abstract}
In this paper we present a formula for the index of a pseudodifferential operator with invertible principal symbol in the {\em extended} Heisenberg calculus of Epstein and Melrose \cites{Me97, Ep04, EMxx}. Our results  build on the work we did in \cite{GvE22}, where we restricted attention to the Heisenberg calculus proper.
\end{abstract}
\maketitle
\setcounter{tocdepth}{1}
\tableofcontents

\section{Introduction}
In this paper we present a formula for the index of a pseudodifferential operator with invertible principal symbol in the {\em extended} Heisenberg calculus of Epstein and Melrose \cites{Me97, Ep04, EMxx}.

The extended Heisenberg calculus includes, as subalgebras, the algebra of Heisenberg pseudodifferential operators {\em as well as} the algebra of classical pseudodifferential operators.
Thus, an index formula for the extended Heisenberg calculus is a common generalization of the  Atiyah-Singer index formula for elliptic operators on the one hand, and the index theorem for Heisenberg elliptic operators on the other hand.
The monograph \cite{EMxx} of Epstein and Melrose aimed to find such a formula, but the project was  unfinished.
Our results here build on the work we did in \cite{GvE22}, where we restricted attention to the Heisenberg calculus proper.
Note that it is a non-trivial exercise to derive the Atiyah-Singer formula in its standard form from our formula.

A hypoelliptic  operator of order zero in the extended Heisenberg calculus on a compact contact manifold $M$ is a  bounded Fredholm operator.
Its symbol  determines an element  in the $K$-theory of the noncommutative algebra of Heisenberg symbols $\Symb_{eH}^{0,0}$.
We construct a character homomorphism
\[ \chi\circ s\colon K_1(\Symb_{eH}^{0,0})\to H^{odd}(M)\]
which assigns a closed differential form to an invertible principal symbol (of order zero).
As in our previous paper \cite{GvE22},
the construction is heavily based on ideas and constructions from cyclic theory.

We show that if $\msL$ is a pseudodifferential operator of  order zero in the extended Heisenberg calculus with invertible principal symbol $\sigma=\sigma_{eH}(\msL)$,
then \begin{equation*} \ind  \msL = \int_M \chi( s(\sigma))\wedge \hat{A}(M).\end{equation*}
(See Theorem \ref{main_theorem}.)

The paper is organized as follows.
In section \ref{symbols} we describe the algebra of principal symbols in the extended Heisenberg calculus and relate them to sections of bundles of Weyl algebras.
In section \ref{symplectic bundle} we describe connections on a bundle of Weyl algebras,
and section \ref{trace} describes a  regularized trace on the fibers of such bundles.
Section \ref{cyclic section} gives a review of cyclic theory and the Chern character in this context.
In section \ref{extension} we describe an extension of the algebra of principal symbols in the extended Heisenberg calculus which is well suited for  cyclic theory.
Section \ref{charmap} contains the description of our characteristic map $\chi$. Finally in the section \ref{formula section} we state and prove the main result of the paper -- the index formula for  order zero hypoelliptic operators in the extended Heisenberg calculus.

\section{Symbols in the extended Heisenberg calculus}\label{symbols}
In this section we  briefly review the structure of the algebra of Heisenberg and extended Heisenberg principal symbols.
We assume the reader has some   familiarity with the Heisenberg calculus \cites{BG88, Ta84}.
We follow the treatment of  Epstein and Melrose \cites{EMxx, Me97}.
This section serves to  fix our notation.
See also \cite{GvE22} for a more detailed discussion  of some of the material in this section.

\subsection{Contact manifolds}
Throughout this paper, $M$ is a smooth closed orientable  manifold of dimension $2n+1$.
We assume that $M$ is equipped with a contact 1-form, i.e. a 1-form $\alpha$ such that $\alpha(d\alpha)^n$ is a nowhere vanishing volume form. Such an $M$ will be called a \emph{contact manifold}. Note that $M$ is oriented by the form  $\alpha(d\alpha)^n$.

The Reeb field $T$ is the vector field on $M$ determined by $d\alpha(T,\,-\,) = 0$ and $\alpha(T)=1$.
We denote by $H\subset TM$  the  bundle of tangent vectors that are annihilated by $\alpha$.
Then  $TM\cong H\oplus \underline{\RR}$, where $\underline{\RR}=M\times \RR$ is the trivial line bundle.

The restriction of the 2-form $\omega:=-d\alpha$ to each fiber $H_p, p\in M$, is a symplectic form,
$\omega_p:=-d\alpha|H_p$.
We assume that a compatible complex structure $J\in \End(H)$ has been chosen, i.e. an endomorphism of $H$ with $J^2=-\mathrm{Id}$, $\omega_p(Jv, Jw)=\omega_p(v,w)$, $\omega_p(Jv,v)\ge 0$.
The combination of the symplectic and complex structures make $H$ into a complex Hermitian vector bundle, which we denote by $H^{1,0}$.
We shall reserve the notation $H$ for the real vector bundle.
We denote by  $P_H$  the $U(n)$ principal bundle of orthonormal frames in $H^{1,0}$.
$P_H$ is equipped with a right action of $U(n)$. The map $v \mapsto \omega(v, \cdot)$  establishes a canonical isomorphism of bundles $H$ and $H^*$, allowing to transfer to $H^*$ complex and Hermitian structures. In particular, we will use it to identify the bundle of orthonormal frames in $(H^*)^{1,0}$  with $P_H$.

Examples of contact manifolds are: the Heisenberg group, boundaries of strictly pseudoconvex domains (e.g. the odd-dimensional unit sphere in $\CC^n$), the cosphere bundle of a manifold.

\subsection{The Weyl algebra}
Let $V$ be a $2n$-dimensional real vector space with symplectic form $\omega$.
The Weyl algebra $\mcW(V,\omega)$
consists of smooth complex-valued functions $a\in C^\infty(V,\CC)$ with properties:
\begin{itemize}
    \item There is an integer $m\in \ZZ$ such that for every  multi-index $\alpha$ there is $C_{\alpha}>0$ with
\begin{equation}\label{seminorms1}
|(\partial^\alpha_v a)(v)| \le C_{\alpha} (1+\|v\|^2)^{\frac{1}{2}(m-|\alpha|)}\qquad \text{for all}\;v\in V.
\end{equation}
The integer $m$ is the Weyl order of  $a$.
\item $a$ has a 1-step polyhomogeneous asymptotic expansion
\begin{equation}\label{asymptotic expansion}
a\sim \sum_{j=-m}^\infty a_j\qquad a_j(sv)=s^{-j}a_j(v) \quad s>0, v\ne 0. \end{equation}
\end{itemize}
The product of two symbols $a, b\in \mcW(V,\omega)$ is
\begin{equation*}  (a \# b)(v) = \frac{1}{(2\pi)^{2n}}\iint e^{2i\omega(x,y)} a(v+x)b(v+y) \,dxdy,
\end{equation*}
where $dx=dy=|\omega^n|/n!$.
The Weyl algebra is $\ZZ$-filtered by the Weyl order. We denote by $\mcW^m$ the space of elements of order $m$,
\begin{equation*} \cdots\subset \mcW^{-2}\subset \mcW^{-1}\subset \mcW^{0}\subset \mcW^{1}\subset \mcW^{2}\subset \cdots\end{equation*}
Schwartz class functions $\mcS(V,\omega)$ form a two-sided  ideal in  $\mcW(V, \omega)$.
Denote the quotient algebra by $\mcB(V, \omega)=\mcW(V, \omega)/\mcS(V,\omega)$,
with quotient map
\begin{equation*} \lambda:\mcW(V, \omega)\to \mcB(V, \omega).
\end{equation*}
Elements of $\mcB(V,\omega)$ can be identified with formal series $\sum_{l=-m}^\infty a_{l}$.
The product $a\star b=c$  of formal series $a= \sum a_{l}$ and $b= \sum b_{m}$ in $\mcB$  is given by the formal series
\begin{equation*}a \star b =  \sum  c_p,
\end{equation*}
where $c_p$ is homogeneous  of degree $-p$, and given by the finite sum
\begin{equation*}c_{p} := \sum_{2k+l+m=p}  B_k(a_{l},b_{m}),
\end{equation*}
\begin{equation*} B_k(f, g)=\left(\frac{i}{2}\right)^k \sum_{|\alpha|+|\beta|=k}\frac{1}{\alpha!\beta!}(-1)^{|\beta|}(\partial^\alpha_x\partial^\beta_\xi f)(\partial^\beta_x\partial^\alpha_\xi g), \end{equation*}
where $x$, $\xi$ are symplectic (Darboux) coordinates on $V$.
We denote by  $\mcB_q(V,\omega)$ the subset of $\mcB(V,\omega)$ consisting of formal series
$a=\sum_{l=-m}^\infty a_{2l}$ for which  all the terms $a_{2l}$ are homogeneous of even degree $-2l$.
$\mcB_q(V,\omega)$ is a subalgebra of $\mcB(V,\omega)$.

\begin{definition}
Let $\iota \colon \mcB_q(V,\omega)\to\mcB_q(V,\omega)$ be the linear map
\begin{equation*}
a=\sum_{j=-m}^\infty a_{2j}\in \mcB_q(V,\omega) \qquad \iota(a):= \sum_{j=-m}^\infty (-1)^ja_{2j}\in \mcB_q(V,\omega), \end{equation*}
where $a_{2j}=a_{2j}(v)$ is homogeneous of degree $-2j$.
\end{definition}
The map $\iota$ is an involution,
\begin{equation*} \iota(a\star b) = \iota(b) \star \iota(a).
\end{equation*}
We let
\begin{equation*}\mcW_q(V,\omega):=\{w\in \mcW(V,\omega)\mid \lambda(w)\in \mcB_q(V,\omega)\},\end{equation*}
\begin{equation*}
\mcW^{2m}_q(V,\omega):=\{w\in \mcW^{2m}(V,\omega)\mid \lambda(w)\in \mcB_q(V,\omega)\}.\end{equation*}
%and
%\begin{equation*} \Alg(V,\omega) := \{(w_+,w_-)\in  \mcW_q(V,\omega)\oplus \mcW_q(V,\omega)^{op}\mid \lambda(w_+) =\iota \circ \lambda(w_-)\}\end{equation*}
%\begin{equation*} \Alg^{2m}(V,\omega) := \{(w_+,w_-)\in  \mcW^{2m}_q(V,\omega)\oplus \mcW^{2m}_q(V,\omega)\mid \lambda(w_+) =\iota \circ \lambda(w_-)\}\end{equation*}
%Here $\mcW_q(V,\omega)^{op}=\mcW_q(V,-\omega)$ is the opposite algebra of $\mcW_q(V,\omega)$, i.e. the product of elements in $\Alg(V,\omega)$ is
%\begin{equation*} (u_+,u_-)(w_+,w_-):= (u_+\#w_+,w_-\# u_-)\end{equation*}
%$\Alg(V,\omega)$ is a $\ZZ$-graded algebra.
\subsection{Action of the symplectic and unitary groups}
As can be seen from the definition of the $\#$ product, a linear symplectic  transformation $\phi \colon V\to V$  determines an automorphism of the algebra
$\mcW (V, \omega)$. For $\phi \in \Sp(V)$, $a\in \mcW(V, \omega)$ let $\phi(a):= a \circ \phi^{-1}$.
Then
\begin{equation*} \phi(a)\# \phi(b) = \phi(a\#b).\end{equation*}
This action preserves the ideal $\mcS(V,\omega)$, and so descends to an action on $\mcB (V, \omega)$. It preserves
$\mcB_q (V, \omega)\subset \mcB (V, \omega)$ and hence $\mcW_q (V, \omega)\subset \mcW (V, \omega)$.

Let $\mathfrak{g}\subset \mcW (V, \omega)$ be the subspace of  homogeneous polynomials of degree $2$ that are purely imaginary (i.e. with values in $i\RR$).
%\begin{equation*}\mathfrak{g}:= \{X \in C^\infty(V; i\mathbb{R})\mid X(s v) = s^2 X(v)\;  \forall s \in \mathbb{R}, v \in V\}\end{equation*}
%%%% Comment Erik Sep 2, 2020: I don't think this is correct
%%%%%%%%%%%%%
$\mathfrak{g}$ is a (real) Lie algebra endowed with the bracket
\[[X,Y]=X\#Y-Y\#X =i\{X, Y\}.\]
%The Lie algebra $\mathfrak{g}$ acts on $\mcW$ by inner derivations:
%\begin{equation*}
%w\mapsto [X, w]=i\{X, w\}\qquad  X\in \mathfrak{g},\; w \in \mcW
%\end{equation*}
%If $w\sim \sum w_m$ then $\{X,w\}\sim \sum \{X,w_m\}$, where $\{X,w_m\}$ is homogeneous of degree $-m$, like $w_m$.
%Therefore the action of $X$ on $\mcW$ preserves $\mcW_q$.
%Note that this action is given by  differentiation along the Hamiltonian vector field of $iX$.
Let $V^*:=\Hom(V, \mathbb{R})\subset \mcW$ be the (real) dual space of $V$.
If   $X\in \mathfrak{g}$ and $f \in V^*$ then  $[X, f] =\{iX,f\}$ is real-valued and homogeneous of degree 1, and so  $[X,f]\in V^*$.
We obtain a morphism of Lie algebras,
\begin{equation*}
\mu^*  \colon  \mathfrak{g} \to \End V^* \qquad \mu^*(X):= [X,\,\cdot\,].
\end{equation*}
The map $V\ni v \mapsto \omega(v, \cdot) \in V^*$ establishes an isomorphism $V \to V^*$ which can be used to define a symplectic form $\omega^*$  on $V^*$.
For $f$, $g \in V^*$ we have
$ \{f, g\}=\omega^*(f, g)\cdot 1$,
where  $1\in \mcW$ is a constant function on $V$.

If  $X \in \mathfrak{g}$ then $\mu^*(X) \in \mathfrak{sp}(V^*)$.
The map $ \mu^*  \colon  \mathfrak{g} \to \mathfrak{sp}(V^*)$ is a Lie algebra isomorphism.
We will identify the groups of symplectic transformations $\mathrm{Sp}(V)\cong \mathrm{Sp}(V^*)$ as well as the corresponding Lie algebras $\mathfrak{sp}(V)\cong \mathfrak{sp}(V^*)$ via the isomorphism $V\to V^*$: $v \mapsto \omega(v, \cdot)$.
We therefore obtain an isomorphism of Lie algebras
\begin{equation}\label{definitionmu} \iso \colon \mathfrak{g}\to \mathfrak{sp}(V).
\end{equation}
The action of $Sp(V)$ on  $\mcW(V,\omega)$ induces an action of the Lie algebra $\mathfrak{sp}(V)$ by derivations,
\begin{equation*} (\phi.w)(v):=-w(\phi(v))\qquad v\in V,\, w\in \mcW(V,\omega),\;\phi \in \mathfrak{sp}(V).
\end{equation*}
This is an action by inner derivations:
\begin{equation*} (\phi.w)(v)=[\mu^{-1}(\phi), w]\qquad  w\in \mcW(V,\omega),\;\phi \in \mathfrak{sp}(V).\end{equation*}
Similarly  $Sp(V)$ and $\mathfrak{sp}(V)$ act on  $\mcW(V,\omega)^{op}$  by automorphisms and derivations respectively. The action of $\mathfrak{sp}(V)$ is again by inner derivations:
\begin{equation*} (\phi.w)(v)=[-\mu^{-1}(\phi), w]_{op}\qquad  w\in \mcW(V,\omega)^{op},\;\phi \in \mathfrak{sp}(V),\end{equation*}
where $[\cdot, \cdot]_{op}$ is the commutator in $\mcW(V,\omega)^{op}$.

\subsection{Bundle of Weyl algebras}
Fix Darboux coordinates $(x, \xi)$, $x=(x_1, \ldots, x_n)$, $\xi=(\xi_1, \ldots, \xi _n)$ on $V$, thus identifying $V \cong \mathbb{R}^{2n}$.
In this situation,  we will omit the vector space from the notations, writing simply $\mcW$ for $\mcW(\mathbb{R}^{2n})$, etc.

Recall that on a compact contact manifold $M$ with contact form $\alpha$, $H=\Ker \alpha\subset TM$ is a symplectic bundle with symplectic form $\omega=-d\alpha$. As before, we denote by $J\in \End(H)$ is a compatible almost complex structure,  and by $P_H$ is the principal $U(n)$-bundle of frames of $(H^*)^{1,0}$.

Since $U(n) \subset \Sp(2n)$  acts on $\mcW_q\subset \mcW$  and $\mcW^{op}_q\subset \mcW^{op}$ by automorphisms, we can form associated bundles:
\begin{equation*}W_H^+:= P_H\times_{U(n)}\mcW_q\quad  W_H^-:= P_H\times_{U(n)}\mcW^{op}_q. \end{equation*}
%We identify $\Alg_H$ with the algebra of smooth sections of the bundle $\mathcal{A}_H$.
% \begin{equation*}\Alg_H = C^\infty(P_H; \Alg)^{Sp(2n)}\end{equation*}
The elements of the algebra of smooth sections of $W_H^+$ can be described as
  smooth $U(n)$-invariant functions from $P_H$ to $\mcW_q$,
\begin{equation*}
\mcW_H: = C^\infty(M;W^+_H) =C^\infty(P_H; \mcW_q)^{U(n)}.
\end{equation*}
where the action of $\phi \in U(n)$ on $s  \colon P_H\to  \mcW_q$ is given by $(\phi \cdot s)(p):= \phi(s(p\phi))$.
In other words, $s$ is invariant if $s(p\phi)=\phi^{-1}(s(p))$. Similarly, the algebra can be described as
\begin{equation*}
\mcW_H^{op} = C^\infty(M;W^-_H) =C^\infty(P_H; \mcW_q)^{U(n)}.
\end{equation*}
Finally, we can construct the bundle
\begin{equation*}B:= P_H\times_{U(n)}\mcB_q \end{equation*}
and the algebra
\begin{equation*}
\mcB_H: = C^\infty(M; B) =C^\infty(P_H; \mcB_q)^{U(n)}.
\end{equation*}
The map $\lambda$ then defines a homomorphism
$\lambda \colon \mcB_H \to \mcW_H$.
%We denote by $\mcW_H$ the algebra of smooth sections in the bundle over $M$ whose fiber at $p\in M$ is the Weyl algebra $\mcW_q(H_p^*,\omega^*_p)$. Here  $\omega_p^*$ is the symplectic form on $H^*_p$ that is dual to $\omega_p$. $\mcW_H^{2m}$ is the subspace of sections in the bundle with fiber $\mcW_q^{2m}(H_p^*,\omega^*_p)$. $\mcW_H^{op}$ is the bundle whose fibers are the opposite algebras $\mcW_q(H^*_p,\omega_p^*)^{op}=\mcW_q(H^*_p,-\omega^*_p)$.
\subsection{Heisenberg principal symbols}

Let $\Psi^m_H$ be the set of Heisenberg pseudodifferential operators of order $m\in \ZZ$.
The algebra of order zero principal symbols,
\begin{equation*} \Symb_H^0:=\Psi^0_H/\Psi^{-1}_H\end{equation*}
is a subalgebra of $\mcW_H\oplus \mcW_H^{op}$,
\begin{equation*} \Symb_H^0 =\{(w_+, w_-)\in \mcW_H^0\oplus (\mcW^0_H)^{op} \mid \lambda(w_+) =\iota \circ \lambda(w_-)\}, \end{equation*}
where $\mcW_H^0$, $(\mcW^0_H)^{op}$ denote the subalgebras of $\mcW_H$ and  $\mcW_H^{op}$ respectively consisting of sections of order $0$ in the Weyl algebra.
\subsection{Principal symbols in the extended Heisenberg calculus}

Let $\Psi^{m,k}_{eH}$ be the set of  pseudodifferential operators in the extended Heisenberg calculus that have  classical order $m\in \ZZ$, and Heisenberg order $k\in \ZZ$ (see \cite{EMxx}).
The algebra of order $(0,0)$ principal symbols,
\begin{equation*} \Symb_{eH}^{0,0}:=\Psi^{0,0}_{eH}/\Psi^{-1,-1}_{eH}\end{equation*}
is
\begin{equation*} \Symb_{eH}^{0,0} =\{(w_+, w_-, f)\},\end{equation*}
where
\begin{itemize}
    \item $w_+ \in \mcW_H^0$, $w_- \in (\mcW^0_H)^{op}$
    \item $f:[0,1]\to C^\infty(S(H^*))$ is a smooth function; here $S(H^*)$ is the sphere bundle of $H^*$.
    \item $f(0)\in C^\infty(S(H^*))$ is the leading term in the formal series  $\lambda(w_+)$, while $f(1)\in C^\infty(S(H^*))$ is the leading term in the formal series $\lambda(w_-)$.
\end{itemize}
The product is
\begin{equation*}(w_+, w_-, f)(u_+, u_-, g) = (w_+\# u_+, u_-\# w_-, fg),\end{equation*}
where $fg$ is pointwise multiplication of the functions $f, g$.

\section{Connection and curvature for a bundle of Weyl algebras}\label{symplectic bundle}

Let  $\Omega^\bullet(P_H; \mcW)_{basic}$ be the space of basic $\mcW$-valued forms.
Recall that an $\mcW$-valued  differential form  $\eta$ on $P_H$ %(of positive degree)
is called basic if:
\begin{itemize}
    \item $\eta$ is horizontal, i.e. $\iota_X \eta=0$ for every vertical vector field $X$ on $P_H$;
    \item $\eta$ is $U(n)$ invariant, i.e. $\phi^* \eta = \phi^{-1}(\eta)$ for every $\phi \in U(n)$.
\end{itemize}
%We have a canonical identification
$k$-forms with values in the bundle $\mcW_H$ are, by definition, basic $\mcW$-valued $k$-forms.
We denote
\begin{equation*}
% \Omega^\bullet(P_H; \mathcal{A}_H)_{basic}\cong \Omega^\bullet(M; \mathcal{A}_H)
 \Omega^k(\mcW_H) = \Omega^k(M;W^+_H) = \Omega^k(P_H; \mcW_q)_{basic}.
\end{equation*}
A unitary connection $\nabla$ on $H\cong H^*$ can be represented by a  connection $1$-form
\begin{equation*} \beta \in  \Omega^1(P_H; \mathfrak{u}(n)). \end{equation*}
The curvature of $\nabla$ is
\begin{equation*} \theta := d\beta +\frac{1}{2}[\beta, \beta] \in \Omega^2(P_H, \mathfrak{u}(n))_{basic}. \end{equation*}
The 1-form $\beta$ defines a  covariant derivative
\begin{equation*}
%\con  \colon  \Alg_H \to  \Omega^1(M; \mathcal{A}_H)
\con^+  \colon  \mcW_H \to  \Omega^1(\mcW_H)
\end{equation*}
by
\begin{equation*}
%\con (a) :=da +\beta \cdot a\in \Omega^1(P_H;\Alg)_{basic} \qquad a \in  \Alg_H = C^\infty(P_H; \Alg)^{Sp(2n)}
\con^+ (a) :=da +\beta \cdot a\in \Omega^1(\mcW_H), \qquad a \in  \mcW_H = C^\infty(P_H; \mcW_q)^{U(n)}.
\end{equation*}
This covariant derivative extends to a derivation
\begin{equation*}%\label{defcon}
%\con  \colon  \Omega^k(M; \mathcal{A}_H) \to \Omega^{k+1}(M; \mathcal{A}_H)
\con^+  \colon  \Omega^k(\mcW_H) \to \Omega^{k+1}(\mcW_H)
\end{equation*}
defined by the same formula:
\begin{equation*}
%\con (\eta) :=d\eta +\beta \cdot \eta  \in \Omega^{k+1}(P_H; \Alg)_{basic}\qquad  \eta\in  \Omega^k(P_H; \Alg)_{basic}
\con^+ (\eta) :=d\eta +\beta \cdot \eta.
\end{equation*}
The curvature of $\con^+$ is defined by
%\begin{equation*}
%\con^2 (\eta)= (d\alpha +\frac{1}{2}[\alpha, \alpha])\cdot \eta =
%[\Iso(d\alpha +\frac{1}{2}[\alpha, \alpha]), \eta]
%\end{equation*}
%\begin{equation*}
%(\con^2 (\eta)= \theta\cdot \eta
%\end{equation*}
%By Proposition ?? in \cite{GvE22},
%\begin{equation}\label{defcon}
%\con (\eta) = d\eta +[\Iso(\beta), \eta] \qquad %\con^2 (\eta) = [\Iso(\theta), \eta] %\qquad \eta %\in  \Omega^k(P_H; \Alg)_{basic}
%\end{equation}
\begin{equation}\label{deftheta}
%\curv:= \Iso (\theta) \in  \Omega^2(P_H, \Alg)_{basic}   \cong \Omega^2(M, \mathcal{A}_H)
\curv^+:= \mu^{-1} (\theta) \in  \Omega^2(\mcW_H)
\end{equation}
Then
\begin{equation*}
    (\con^+)^2(\eta) = [\curv^+,\eta]\qquad  \eta\in \Omega^\bullet(\mcW_H).
\end{equation*}
%for any $\eta \in  \Omega^k(M;\mathcal{A}_H) = \Omega^k(P_H; \Alg)_{basic}$.

\begin{lemma} With the definitions above we have
\begin{equation*}
    \con^+(\curv^+)=0.
\end{equation*}
\end{lemma}

In an entirely analogous manner, connection form $\beta$ yields a covariant derivative
\begin{equation*}%\label{defcon}
%\con  \colon  \Omega^k(M; \mathcal{A}_H) \to \Omega^{k+1}(M; \mathcal{A}_H)
\con^-  \colon  \Omega^k(\mcW_H^{op}) \to \Omega^{k+1}(\mcW_H^{op}),
\end{equation*}
where
\begin{equation*}
\Omega^k(\mcW_H^{op}) = \Omega^k(M;W^-_H) = \Omega^k(P_H; \mcW_q^{op})_{basic}.
\end{equation*}
It satisfies
\begin{equation*}
    (\con^-)^2(\eta) = [\curv^-,\eta]\qquad  \eta\in \Omega^\bullet(\mcW_H^{op}),
\end{equation*}
where $\curv^{-} =-\curv^+ \in \Omega^2(\mcW_H^{op})$

Finally, the same construction defines a covariant derivative
\begin{equation*}%\label{defcon}
%\con  \colon  \Omega^k(M; \mathcal{A}_H) \to \Omega^{k+1}(M; \mathcal{A}_H)
\con^+  \colon  \Omega^k(\mcB_H) \to \Omega^{k+1}(\mcB_H),
\end{equation*}
where
\begin{equation*}
% \Omega^\bullet(P_H; \mathcal{A}_H)_{basic}\cong \Omega^\bullet(M; \mathcal{A}_H)
 \Omega^k(\mcB_H) = \Omega^k(M;B) = \Omega^k(P_H; \mcB_q)_{basic}
\end{equation*}
satisfying
\begin{equation*}
    (\con^+)^2(\eta) = [\curv^+,\eta]\qquad  \eta\in \Omega^\bullet(\mcB_H)
\end{equation*}
and
\begin{equation*}
    \lambda(\con^+(\eta) =  \con^+(\lambda(\eta))\qquad  \eta\in \Omega^\bullet(\mcW_H).
\end{equation*}

\section{A regularized trace for the bundle of Weyl algebras}\label{trace}
Weyl quantization  associates with $a \in \mcW=\mcW(\RR^{2n})$ a psedudiffferential operator $A=\Op^w(a)$  defined (formally) on functions $u\in \mcS(\RR^n)$ as
\begin{equation}\label{eqn:Weyl}
 (Au)(x) = \frac{1}{(2\pi)^n} \iint e^{i(x-y)\cdot \xi} a\left(\frac{x+y}{2},\xi\right)u(y)\,dy\, d\xi.
 \end{equation}
 We denote by $\ho$ the harmonic oscillator
\begin{equation*}
\ho= \sum \limits_{j=1}^n \left(-\frac{\partial^2}{\partial x_j^2}+x_j^2\right)= \Op^w\left(\sum \limits_{j=1}^n (\xi_j^2+x_j^2)\right).
\end{equation*}
$\ho$ is a strictly positive unbounded selfadjoint operator on $L^2(\RR^n)$.

For $a\in \mcW_q^{2m}$ of even order $2m$, $\Op^w(a)\ho^{-z}$ is of  trace class if $\re z> n+m$.
Define the zeta-function
\begin{equation*} \zeta_a(z) := \Tr(\Op^w(a)\ho^{-z}), \quad  \re z> n+m.\end{equation*}
%The complex powers $Q^{-z}$ and product $wQ^{-z}$ are calculated in the representation of the Weyl algebra as operator on $L^2(\RR^n)$.
The zeta function is holomorphic for $\re z> n+m$, and it
extends to a meromorphic function with at most simple poles at $m+n, m+n-1, m+n-2, \dots$.
The residue at $z=0$  of the  zeta-function gives a residue trace on $\mcW_q$,
\begin{equation*} \Res \colon \mcW_q\to \CC\qquad \Res(a) = \lim_{z\to 0} z\zeta_a(z). \end{equation*}
$\Res$ is a trace on $\mcW_q$ that vanishes on the Schwartz class ideal $\Sch$.
It follows that residue induces a trace on the quotient $\mcB_q=\mcW_q/\Sch$ which we also denote $\Res$.
An explicit formula for  $\Res a$ in terms of the asymptotic expansion $a=\sum_ja_{2j}\in \mcB_q$ is
\begin{equation*}
\Res a= -\frac{1}{2(2\pi)^n}\int_{S^{2n-1}}a_{2n}(\theta) d \theta,
\end{equation*}
where $S^{2n-1}$ is the unit sphere in $V$.

We denote by $\Trh (a)$ the constant term at $z=0$ of the zeta-function,
\begin{equation*}  \Trh(a) = \lim_{z\to 0}  \left(\zeta_a(z) - \frac{1}{z}\Res(a)\right).
\end{equation*}
If $a\in \mcS$ then $\Tr(\Op^w(a)\ho^{-z})$ is an entire function, and we see that
\begin{equation*} \Trh(a) = \Tr \Op^w(a) =\frac{1}{(2\pi)^n} \int a(x,\xi) dx \, d\xi\qquad \forall a\in \Sch(\mathbb{R}^{2n}).\end{equation*}
However the functional $\Trh$ is not a trace on $\mcW_q$. More exlicitely, we have the following result. First note that even though  $\log(x^2+\xi^2) \notin \mcB_q$, the formula
\begin{equation}\label{defdelta}
\delta(a) = a\star \log(x^2+\xi^2) - \log(x^2+\xi^2) \star a
\end{equation}
defines   a derivation $\delta$ of $\mcB_q$.
Direct calculation shows the following
\begin{lemma}
%\begin{equation*}
$\Trh([a, b]) =\Res \lambda(a)  \delta (\lambda(b))$.
%\end{equation*}
\end{lemma}
The regularized trace $\Trh$ is \emph{not} invariant under the action of the symplectic group $\Sp(2n):=\Sp(\mathbb{R}^{2n})$. However it is invariant under the action of unitary subgroup $U(n) \subset \Sp(2n)$:
\begin{equation*}
    \Trh(\phi(a))=\Trh(a) \qquad   a\in \mcW,\;\phi \in U(n).
\end{equation*}
Clearly the same invariance holds for $\Trh$ on $\mcW^{op}$.
Therefore we have regularized traces defined on sections in the bundles of Weyl algebras,
\[\Trh : \mcW_H\to C^\infty(M)\qquad \Trh:\mcW_H^{op}\to C^\infty(M).\]

\section{The Chern character in cyclic homology}\label{cyclic section}
In this section we give a very brief overview of the periodic cyclic homological complex, mostly to fix the notation. We also recall the  Chern character map into the cyclic homology.
The standard reference for this material is \cite{loday}.

For a complex unital algebra $A$ set $C_l(A):= A \otimes(A/(\mathbb{C} \cdot 1))^{\otimes l}$, $l \ge 0$.
%For a topological algebra, e.g. a normed algebra, one needs to take an appropriately completed tensor product.
One defines differentials $b \colon  C_l(A) \to C_{l-1}(A)$ and $B  \colon  C_l(A) \to C_{l+1}(A)$ by
\begin{align*}
b (a_0\otimes a_1\otimes \ldots a_l):=& \sum_{i=0}^{l-1} (-1)^ia_0\otimes \ldots a_ia_{i+1}\otimes \ldots a_l+(-1)^la_la_0\otimes a_1\otimes \ldots a_{l-1},\\
B (a_0\otimes a_1\otimes \ldots a_l):=& \sum_{i=0}^{l} (-1)^{li}1 \otimes a_i\otimes a_{i+1}\otimes \ldots a_{i-1}\text{ (with $a_{-1}:=a_l).$}
\end{align*}
One verifies directly that $b$, $B$ are well defined and satisfy $b^2=0$, $B^2=0$, $Bb+bB=0$.  Let $u$ be a formal variable of degree $-2$.
The space of periodic cyclic chains of degree $ i \in \ZZ$ is defined by
\begin{equation*}
CC^{per}_{i}(A) = \left(C_{\bullet}(A)[u^{-1},u]]\right)_i=\prod \limits_{-2n+l=i} u^nC_l(A).
\end{equation*}
Note that $CC^{per}_{i}(A)= uCC^{per}_{i+2}(A)$.
We will write a chain in $CC^{per}_i(A)$ as
$ \alpha = \sum \limits_{i+2m\ge 0}u^m \alpha_{i+2m}$ where $\alpha_l \in C_l(A)$.
The boundary is given by $b+uB$ where $b$ and $B$ are the Hochschild and Connes boundaries of the cyclic complex.
%The homology of this complex is denoted $HC^{per}_{\bullet}(A)$.
The homology of this complex is periodic cyclic homology, denoted $HC^{per}_{\bullet}(A)$.

If $r \in M_n(A)$ is invertible the following formula defines a cycle in the periodic cyclic complex:
\begin{equation}\label{cyclicchernformula}
\Ch(r):= -\frac{1}{2 \pi i} \sum_{l=0}^{\infty}(-1)^l\, l!\, u^l \tr (r^{-1}\otimes r)^{\otimes (l+1)} \in CC^{per}_1(A),
\end{equation}
where $\tr  \colon  (A \otimes M_n(\mathbb{C}))^{\otimes k} \to A^{\otimes k}$ is the map given by
\begin{equation*}
\tr (a_0\otimes m_0)\otimes (a_1\otimes m_1) \otimes \ldots (a_k\otimes m_k)= (\tr m_0m_1\ldots m_k) a_0  \otimes a_1 \otimes \ldots a_k.
\end{equation*}
In the case of interest to us,  $A$ will be a Fr\'echet algebra.
In that case we will use a projective tensor product to define $C_l(A)=A \otimes(A/(\mathbb{C} \cdot 1))^{\otimes l}$.
One can define the topological $K$-theory of a Fr\'echet algebra as $K_1(A):=\pi_0(GL(A))$.
With these definitions, \eqref{cyclicchernformula} defines  the (odd) Chern character homomorphism
\begin{equation}\label{CyclicChern}
%\Ch  \colon  K_1(A) \to HC^{per}_1(A)
\Ch  \colon  K_1(A) \to HC^{per}_1(A)
\end{equation}
from  topological $K$-theory to  periodic cyclic homology.
If $t \mapsto r_t$, $t\in [0,1]$ is a (piecewise) smooth map from $[0,1]$ to  $GL_n(A)$ then
\begin{equation*}
\Ch (r_1)-\Ch(r_0) = (b+uB) \Tch(r_t),
\end{equation*}
where
\begin{equation*}
\Tch (r_t):= -\frac{1}{2 \pi i}\int_{0}^{1} \xi(t) dt   ,
\end{equation*}
\begin{equation*}  \xi(t)= \tr (r_t^{-1}\frac{dr_t}{dt})
 +\sum_{l=0}^\infty (-1)^{l+1} l! u^{l+1}
 \sum_{j=0}^l \tr \big( (r_t^{-1} \otimes r_t)^{\otimes j+1}\otimes
   r_t^{-1} \frac{dr_t}{dt} \otimes (r_t^{-1} \otimes r_t)^{ \otimes l-j} \big),
\end{equation*}
which shows in particular that Chern character is well-defined.

\section{An extension of the algebra of symbols}\label{extension}
In this section we define an extension of the algebra of extended Heisenberg symbols,  algebra $\Ext_H$, which is a  convenient source of the characteristic map in the next section.

Let $\mcB_H$ be the algebra of smooth sections in the bundle over $M$ whose fiber at $p\in M$ is $\mcB_q(H_p^*,\omega_p^*)$.
We let $\Ext_H$ be the algebra
\begin{equation*} \Ext_H := \{(w_+,w_-,r) \},\end{equation*}
where
\begin{itemize}
    \item $w_+ \in \mcW_H$, $w_- \in \mcW_H^{op}$.
    \item $r:[0,1]\to \mcB_H$ is a smooth function.
    \item $r(1) = \lambda(w_+)$, $r(0)=\iota \circ \lambda(w_-)$.
\end{itemize}
There is the obvious  algebra homomorphism,
\begin{equation*} \Ext_H \to \Symb_{eH}^{0,0}\quad (w_+,w_-, r)\mapsto (w_+, w_-, f),\end{equation*}
where $f(t)$ is the leading term of the series $r(t)\in \mcB_H$.
This quotient map has a canonical section,
\begin{equation*} s\colon  \Symb_{eH}^{0,0}\to \Ext_H \quad (w_+,w_-, f)\mapsto (w_+, w_-, r(t))\end{equation*}
%where $r(t)$ is a series  for which all  terms vanish except for the leading term, which is $f(t)$.
where $r(t)=r_0+tr_1$ is linear in $t$,
$r_0=\lambda(w_-)$, $r_1=  \lambda(w_+)-\iota \circ\lambda(w_-)$.
While this section $s$ is not an algebra homomorphism,  it maps invertible elements to invertible elements and gives a well-defined map of homotopy classes of invertible elements.
Thus, this section $s$ determines a map in $K$-theory,
\[ K_1(\Symb_{eH}^{0,0})\to K_1(\Ext_H).\]
A pseudodifferential  operator $\msL\in \Psi_{eH}^{0,0}$ with invertible principal symbol $\sigma_{eH}(\msL)\in \Symb_{eH}^{0,0}$ determines a class
\[ [\sigma_{eH}(\msL)]\in K_1(\Symb_{eH}^{0,0})\]
and, via $s$,  also a class
\[ [s(\sigma_{eH}(\msL))]\in K_1(\Ext_H).\]
We shall express the index of $\msL$ as a function of this symbol class in $K_1(\Ext_H)$.

\section{The characteristic map}\label{charmap}
In this section, we associate with the algebra $\Ext_H$ the complex $\mcC_\bullet(\Ext_H)$ and define a characteristic map from this complex to de Rham complex.

As a vector space,
\begin{equation*}
\mcC_\bullet(\Ext_H):= CC^{per}_\bullet(\mcW_H) \oplus CC^{per}_\bullet(\mcW_H^{op}) \oplus CC^{per}_{\bullet+1}(\mcB_H).
\end{equation*}
The differential is given by
\begin{equation}\label{defboundary}
\partial (\zeta^+, \zeta^-, \gamma) =
\left((b+uB)\zeta^+,(b+uB)\zeta^-, \lambda_*(\zeta^+) -(\iota\circ\lambda)_*(\zeta_-)) -(b+uB) \gamma \right),
\end{equation}
where $(\zeta^+, \zeta^-, \gamma) \in \mcC_\bullet(\Ext_H)$.
The homology of the complex $\mcC_\bullet(\Ext_H)$ is denoted by $\HH_\bullet(\Ext_H)$, and the homology of the cycle $(\zeta^+, \zeta^-, \gamma) \in \mcC_\bullet(\Ext_H)$ is denoted by $[\zeta^+, \zeta^-, \gamma] \in \HH_\bullet(\Ext_H)$.
We have a group homomorphism
\begin{equation*}
\Ch \colon K_1 (\Ext_H) \to \HH_1(\Ext_H)
\end{equation*}
given by
\begin{equation}\label{cherne}
\Ch \colon [(w_+,w_-,r_t)] \mapsto \left[\Ch(w_+),\Ch(w_-),\Tch(r_t)\right].
\end{equation}

We will now define a morphism of complexes
\begin{equation*}
\mcC_\bullet(\Ext_H) \to \Omega^\bullet(M)[u^{-1},u].
\end{equation*}
We first construct maps
\begin{equation*}
\Phi^{+} \colon CC_\bullet(\mcW_H)  \to \Omega^\bullet(M), \quad
\Phi^{-} \colon CC_\bullet(\mcW_H^{op})  \to \Omega^\bullet(M)[u]
\end{equation*}
and
\begin{equation*}
\phi \colon CC_\bullet(\mcB_H) \to \Omega^{\bullet-1}(M)[u]
\end{equation*}
described as follows. Choose a unitary connection on $H$. As in the Section \ref{symplectic bundle} we can define connections $\con^+, \con^-$, with curvatures  $\curv^+, \curv^-=  -\curv^+$.
The definition of the maps $\Phi^{\pm}$ is inspired by \cites{Co94, gor1, gor2} and is given by
 \begin{multline*}
\Phi^\pm(a_0\otimes a_1\otimes a_k)=\\
    \sum_{m=0}^k (-1)^{m(k-m)}\sum_{i_0, i_1, \ldots i_{k+1}} \frac{(-1)^{i_0+\ldots +i_{k+1}}}{(i_0+i_1+\ldots +i_{k+1}+k+1)!}\Trh  (\curv^{\pm})^{i_0} \con^{\pm}(a_m) (u\curv^{\pm})^{i_1}\con^{\pm}(a_{m+1})\ldots \\
(u\curv^{\pm})^{i_{k+1-m}} a_0 (u\curv^{\pm})^{i_{k+2-m}}\ldots \con^{\pm}(a_{m-1})(u\curv^{\pm})^{i_{k+1}},
\end{multline*}
%where
%\begin{equation*}
%\frac{(-1)^{i_0+\ldots +i_k}}{(i_0+i_1+\ldots i_k+k)!}
%\end{equation*}

or equivalently
\begin{multline*}
\Phi^\pm(a_0\otimes a_1\otimes a_k)=\\
\sum_{m=0}^k (-1)^{m(k-m)} \int_{\Delta^{k+1}}\Trh e^{-t_0 u\curv^{\pm}} \con^{\pm}(a_m) e^{-t_1 u\curv^{\pm}}\con^{\pm}(a_{m+1})\ldots \\
e^{-t_{k+1-m}u \curv^{\pm}}a_0 e^{-t_{k+2-m} u\curv^{\pm}}\ldots \con^{\pm}(a_{m-1})e^{-t_{k+1} u\curv^{\pm}}) dt_1\ldots dt_{k+1}.
\end{multline*}
The map $\phi$ is given
\begin{multline*}
\phi(b_0\otimes b_1\otimes b_k)=
\sum_{m=1}^k \sum_{i_0, i_1, \ldots i_{k}} \frac{(-1)^{i_0+\ldots +i_k+m-1}}{(i_0+i_1+\ldots i_k+k)!}\\
\Res  b_0  (u\curv^{+})^{i_0}\con^{+}(b_1)(u\curv^{+})^{i_1} \ldots
 \delta(b_m)\ldots \con^{+}(b_k)(u\curv^{+})^{i_{k}}, \end{multline*}
where $\delta$ is the derivation defined in \eqref{defdelta}.
\begin{proposition}
The maps $\Phi^\pm$, $\phi$ defined above satisfy the following relations:
\begin{align*}
    \Phi^+ \circ (b+uB) - (ud)\circ \Phi^+ &= \phi \circ \lambda_*,\\
    \Phi^- \circ (b+uB) - (ud)\circ \Phi^- &= (-1)^{n}\phi \circ (\iota \circ \lambda)_*,
    \\ \phi \circ (b+uB) - (ud)\circ \phi &=0.
\end{align*}

\end{proposition}
\begin{proof}
These equalities follow from straightforward calculations.

Since the connection is unitary and $\Trh$ is unitarily invariant, we have
\[\Trh(\con(\cdot)) = d \Trh(\cdot)\qquad \Trh [\curv, \cdot]=0.\]
This is used in the first calculation, as well as the equality $\Trh[A, B] = \Res A\delta(B)$.

The second identity follows from the first, using
\[ \con^-(\iota(b)) = \iota (\con^+(b)), \quad \iota(\curv^-)=\curv^+, \quad \delta(\iota(b))=-\iota(\delta(b)),\quad \iota(a)\star_-\iota(b) = \iota(a\star b).\]
The third equality also follows from the first, since $\lambda$ is surjective.

\end{proof}
\begin{theorem}\label{defee}
The map $\Phi\colon \mcC_\bullet(\Ext_H) \to \left(\Omega^\bullet(M)[u], ud \right)$
given by
\begin{equation*}
\Phi(\xi^+,\xi^-, \gamma):=\Phi^+(\xi)-(-1)^n\Phi(\xi^-) - \phi(\gamma)
\end{equation*}
is a morphism of complexes.
\end{theorem}
\begin{proof}
This is an immediate consequence of the identities in the preceding proposition, as well as the definition of the boundary \eqref{defboundary} in the complex  $\mcC_\bullet(\Ext_H)$
\end{proof}
\begin{definition}\label{def:character_map}
We define the characteristic map
\begin{equation*}
\chi \colon K_1(\mcE_H) \to H^{odd}(M)
\end{equation*}
as the composition
\begin{equation*}
  K_1(\mcE_H) \overset{\Ch}{\longrightarrow} \HH_1(\mcE_H) \overset{\Phi}{\longrightarrow} H^{odd}(M)[u^{-1}, u] \overset{R}{ \longrightarrow} H^{odd}(M).
\end{equation*}
Here the first map is given in equation \eqref{cherne}, the second is gven in Theorem \ref{defee}, and
$R(\sum u^q \alpha_{q}) = \sum (2 \pi i)^{-q}\alpha_{q}$, $\alpha_{q} \in H^\bullet(M)$.
\end{definition}

\section{The index formula}\label{formula section}
Let $\msL \in \Psi^{0,0}_{eH}$ be an operator in the extended calculus with invertible principal symbol
\[\sigma_{eH}(\msL) \in \Symb_{eH}^{0,0}:=\Psi^{0,0}_{eH}/\Psi^{-1,-1}_{eH}.\]

\begin{theorem}\label{main_theorem}
Let $M$ be a closed smooth manifold of dimension $2n+1$ with contact form $\alpha$.
We orient  $M$ by the volume form $\alpha(d\alpha)^n$.
Let
\begin{equation*}\msL \colon C^\infty(M,\CC^r)\to C^\infty(M,\CC^r)\end{equation*}
be a  pseudodifferential operator of  order zero in the extended Heisenberg calculus
that acts on sections in a trivial bundle $M\times \CC^r$. Assume that
 $\sigma_{eH}(\msL)\in M_r(\Symb_{eH}^{0,0})$
 %its  extended Heisenberg principal  symbol
 is invertible. Then the index of $\msL$ is
\begin{equation*} \ind  \msL = \int_M \chi( s(\sigma_{eH}(\msL)))\wedge \hat{A}(M).\end{equation*}
Here $s:\Symb^{0,0}_{eH}\to \Ext_H$ is the canonical section, and $\chi:K_1(\Ext_H)\to H^{odd}(M)$ is the character homomorphism of Definition \ref{def:character_map}.
\end{theorem}
\begin{proof}
Both $\ind  \msL$ and
$\int_M \chi( s(\sigma_{eH}(\msL)))\wedge \hat{A}(M)$ depend only on the class
$[\sigma_{eH}(\msL)]\in K_1(\Symb_{eH}^{0,0})$.
Thus we need to verify that the two resulting group homomorphisms
$K_1(\Symb_{eH}^{0,0}) \to \mathbb{C}$ coincide.

The embedding $\Symb_{H}^0 \to \Symb_{eH}^{0,0}$ induces an isomorphism in $K$-theory.
This can be seen by considering the diagram in $K$-theory that corresponds to the following  commutative diagram,
\[\xymatrix{
0\ar[r] &\Sch(H^*,\omega)\oplus\Sch(H^*,-\omega)\ar[r]\ar[d]^{=} &\Symb_H^0\ar[r]\ar[d]^{\subset} &C^\infty(S(H^*))\ar[r]\ar[d] &0\\
0\ar[r] &\Sch(H^*,\omega)\oplus\Sch(H^*,-\omega)\ar[r] &\Symb_{eH}^{0,0}\ar[r] &C^\infty(S(H^*)\times [0,1])\ar[r] &0
}
\]
Here $\Sch(H^*,\omega)$ is the algebra of smooth sections in the bundle whose fibers are the Schwartz class ideals $\Sch(H^*_p,\omega^*_p)$, and $S(H^*)$ is the unit sphere bundle of $H^*$. The map $C^\infty(S(H^*))\to C^\infty(S(H^*)\times [0,1])$ is determined by the homotopy equivalence $S(H^*)\times [0,1]\to S(H^*)$.
The rows in the diagram are short exact sequences.

We see that it  suffices to verify that the homomorphism
\[K_1(\Symb_{eH}^{0,0})\to \CC\qquad \sigma\mapsto \int_M \chi(s(\sigma))\wedge\hat{A}(M)\]
computes the index of  operators in the Heisenberg calculus with invertible principal symbol in $\Symb_H^0\subset \Symb_{eH}^{0,0}$. But in this case our formula reduces to Theorem 9.1 of \cite{GvE22}.
\end{proof}

%\bibliographystyle{amsplain}

%\bibliography{MyBibfile}

\end{document}